\documentclass[a4paper,11pt, oneside]{amsart}
\usepackage{hyperref}

\usepackage[usenames,dvipsnames,svgnames,table]{xcolor}
\usepackage[all]{xy}
\usepackage{enumitem}
\usepackage[normalem]{ulem}
\newcommand{\tsk}[1]{\textcolor{YellowOrange}}

\usepackage{tikz}
\usepackage{tikz-cd}
\usepackage{graphicx}
\usepackage{pgf,tikz}
\usetikzlibrary{arrows}
\usepackage[top=1.2in,bottom=1.2in,left=1.in,right=1.in]{geometry} 

\makeatletter
\def\@endtheorem{\endtrivlist}
\makeatother

\usepackage[active]{srcltx}
\usepackage{amsmath}
\usepackage{amssymb}
\usepackage{amscd}
\usepackage{amsthm}
\usepackage{mathrsfs}  

\usepackage{nicefrac}

\newtheorem{teo}{Theorem}[section]
\newtheorem*{teo*}{Theorem} 

\newtheorem*{mainteoA}{Theorem A}
\newtheorem*{mainteoB}{Theorem B}
\newtheorem*{mainteoC}{Theorem C}

\newtheorem{defin}[teo]{Definition}
\newtheorem{defin*}{Definition}
\newtheorem{prop}[teo]{Proposition}
\newtheorem*{prop*}{Proposition} 
\newtheorem{cor}[teo]{Corollary}
\newtheorem*{cor*}{Corollary} 
\newtheorem{lemma}[teo]{Lemma}

\theoremstyle{definition}

\newtheorem{remark}[teo]{Remark}

\newtheorem*{conj*}{Conjecture} 

\newtheoremstyle{dico}
{\baselineskip}   
{\topsep}   
{}  
{0pt}       
{} 
{.}         
{5pt plus 1pt minus 1pt} 
{}          
\theoremstyle{dico}
\newtheorem{say}[teo]{}
\numberwithin{equation}{section}

\newcommand{\ra}{\rightarrow}
\newcommand{\C}{\mathbb{C}}
\newcommand{\R}{\mathbb{R}}
\newcommand{\Zeta}{{\mathbb{Z}}}

\newcommand{\QQ}{{\mathbb{Q}}}
\newcommand{\meno}{^{-1}}

\newcommand{\Prym}{\operatorname{Prym}}
\newcommand{\restr}[1]          {\vert_{#1}}

\renewcommand{\setminus}{-}

\renewcommand{\phi}{\varphi}

\newcommand{\Gl}{\operatorname{GL}}

\newcommand{\PP}{\mathbb{P}^1}   
\newcommand{\V}{\mathbb{V}}

\newcommand{\OO}{\mathcal{O}}

\newcommand{\sieg}{\mathfrak{S}}
\newcommand{\siegg}{\mathfrak{S}_g}

\newcommand{\U}{\operatorname{U}}
\newcommand{\SU}{\operatorname{SU}}
\newcommand{\Sp}{\operatorname{Sp}}

\newcommand{\A}{\mathcal{A}}
\newcommand{\Z}{\mathsf{Z}}
\newcommand{\M}{\mathcal{M}}

\newcommand{\ag}{\mathcal{A}_g}
\newcommand{\Bb}{\mathcal{B}}
\newcommand{\mg}{\mathcal{M}_g}

\newcommand{\B}{\mathsf{B}}
\newcommand{\mihi}[1]{}

\newcommand{\Nm}{\operatorname{Nm}}
\newcommand{\Pic}{\operatorname{Pic}}
\newcommand{\Mon}{\operatorname{Mon}}
\newcommand{\Hdg}{\operatorname{Hg}}

\begin{document}

\pagestyle{myheadings}

\title{Families of cyclic curve coverings  with maximal monodromy}

\author{Irene Spelta}
\address{Irene Spelta, Institut f\"ur Mathematik, Humboldt-Universit\"at zu Berlin, Unter den Linden 6, 10099 Berlin,
Germany.}
\email{irene.spelta@hu-berlin.de}
\author{Carolina Tamborini}
\address{Carolina Tamborini, Essener Seminar f\"ur Algebraische Geometrie und Arithmetik, Universit\"at Duisburg-Essen, Thea-Leymann-Str. 9, 45127 Essen, Germany.}
\email{carolina.tamborini@uni-due.de}

\thanks{\textit{2020 Mathematics Subject Classification}. 14H10, 14H40, 14G35, 14D07.\\
	The authors are members of GNSAGA (INdAM). I. Spelta was partially supported by the Spanish MINECO research project PID2023-147642NB-I00.
    C.Tamborini was partially supported by the DFG-Research Training Group 2553
“Symmetries and classifying spaces: analytic, arithmetic, and derived”
    }
    
\begin{abstract}
    We study the algebraic monodromy of families of cyclic Galois coverings of curves. Under a condition on the $G$-decomposition of the associated variation of Hodge structures, we prove a criterion for the  maximality of the monodromy. The proof combines the genus-zero case with a degeneration argument involving Prym varieties of certain admissible coverings.  As a consequence of our criterion, we show that for $g\geq 8$ there exists no special family of Galois covers of the type we consider, providing new evidence towards the Coleman-Oort conjecture. Finally, we determine when the loci of double and triple Galois covers are totally geodesic.
\end{abstract}

\maketitle


\section{Introduction}
\begin{say}
In this paper, we prove a criterion for maximality of the algebraic monodromy group for certain families of cyclic coverings of curves. Our technique relies on the study of the natural $\QQ$-variation of Hodge structures associated with the family. 
For a cyclic group $G$, let  $f:\mathcal{C}\rightarrow B$ be the family parametrizing all Galois covers
$C_b \to C'_b = C_b/G$ with prescribed ramification and monodromy. The $G$-action on the curves induces a $G$-action on $H^1(C_b,\QQ)$, which preserves the cup product and yields the inclusion $G\subset \operatorname{Sp}(H^1(C_b,\QQ))$. 
Let $\Mon^0 \subset \operatorname{Sp}(H^1(C_b,\QQ))$ be the algebraic monodromy group associated with the semisimple local system $\mathbb{V}=R^1f_*\QQ$. Since the monodromy must preserve the cup product as well as the $G$-action, we have the inclusion \begin{equation*}
    \Mon^0 \;\subseteq\; (\Sp(H^1(C_b,\QQ))^G)^{der},
\end{equation*}
where $\Sp(H^1(C_b,\QQ))^G$ is the centralizer of $G$ inside $\Sp(H^1(C_b,\QQ))$, and $(\Sp(H^1(C_b,\QQ))^G)^{der}$ denotes its derived subgroup. It is natural to investigate when such inclusion is actually an equality. 
To address this problem, we exploit the following   
decomposition induced by the group $G$: 
\begin{gather}\label{spG}
		\Sp(H^1(C_b, \R))^G \simeq  \prod_{(\chi, \bar\chi ), \chi\neq\bar\chi} \U(m_\chi, m_{\bar \chi})\times\prod_{(\chi, \bar\chi ), \chi=\bar\chi}\Sp(2m_\chi, \R),
	\end{gather}
where $m_{\chi}$ is the multiplicity of the $\chi$-eigenspace for the action of $G$ on $H^{1,0}(C_b)$ (and similarly for the conjugate $\bar\chi$). In the following, we consider families with \emph{no repeating factors} in the monodromy, meaning that, for any two characters $\chi_1$ and $\chi_2$ with $m_{\chi_i}, m_{\bar\chi_i} \neq 0$, the multiplicities satisfy $
\{m_{\chi_1},\, m_{\bar\chi_1}\} \neq \{m_{\chi_2},\, m_{\bar\chi_2}\}
$. Our first result is the following.

\begin{mainteoA}\label{teo A}
 Let $\mathcal{C}\rightarrow \B$ be a family of cyclic $G$-coverings with no repeating factors in the monodromy with $g'=g(C_b')\geq 1$. Then the monodromy group is maximal, namely
\begin{gather*}
        \Mon^{0}_{\R} = (\Sp(H^1(C_b, \R))^G)^{der}.
    \end{gather*}  
\end{mainteoA}
(See Theorem \ref{main}).
Maximality results for the monodromy of families of algebraic varieties are widely studied in algebraic geometry. Regarding $G$-coverings of curves, we mention e.g. \cite{biswasx2, biswasx3, Looijenga, landesmaneco}.  In particular we point out that, for $g'\geq 4$, Theorem A partially recovers \cite[Theorem 1.3]{landesmaneco}, where the authors prove a strong maximality result for the monodromy of families of any Galois covers, requiring some lower bound on the genus of the base curve $C'$. Our argument allows to take into account lower genus $g'$, too. In fact, our proof proceeds by induction on the genus $g'$. In the case $g'=0$, we show that $(\operatorname{Sp}^G)^{pos}\subseteq\Mon^0_{\R}$, where $(\operatorname{Sp}^G)^{pos}$ involves only the contributions in $(\Sp^G)^{der}$ coming from positive-dimensional eigenspaces (see Theorem \ref{tutti diversi}). The induction step is based on a degeneration argument and requires the study of Prym varieties associated with certain admissible coverings, which is carried out in Theorem \ref{iso prym}.  If $g'\geq 1$, we show that $(\Sp^G)^{der}=(\Sp^G)^{pos}$, yielding Theorem A. As a consequence, under the assumption of cyclic coverings with no repeating factors, our argument applies to all genera $g'\geq 1$, and in the case $g'\geq 4$ it provides an alternative proof of the maximality of the monodromy discussed in \cite{landesmaneco}.
\end{say}

\begin{say}
The maximality of the monodromy is a powerful property and it has strong consequences on the geometry of the moduli image of our families. In particular, our result has interesting applications to the Coleman-Oort conjecture. Denote by $\M_g$ the moduli space of smooth projective curves of genus $g$ over $\mathbb{C}$, by $\mathcal{A}_g$ the moduli space of $g$-dimensional principally polarized abelian varieties, and let $j: \mathcal{M}_g \to \mathcal{A}_g$ be Torelli map sending curves to their Jacobians. The \emph{Coleman-Oort conjecture} concerns the interaction between the image of the Torelli map and the special subvarieties of $\mathcal{A}_g$, which are defined as Hodge loci associated with the natural rational variation of Hodge structure on $\mathcal{A}_g$. Special subvarieties are totally geodesic, namely they are linear with respect to the locally symmetric metric induced from the Siegel space. The conjecture predicts that for $g \gg 0$, there are no positive-dimensional special subvarieties $S \subset \mathcal{A}_g$ such that $S$ is generically contained in the Torelli locus, i.e. $S \subset \overline{j(\mathcal{M}_g)}$ and $S \cap j(\mathcal{M}_g) \neq \emptyset$. However, for low genus, counterexamples are known: there exist positive dimensional special subvarieties $S\subset \mathcal{A}_g$ that lie generically in the Torelli locus. All known examples occur for $g \leq 7$, and arise from families of Galois covers of curves. Several such examples have been constructed and studied in the literature; see for instance \cite{dejong-noot, rohde, moonen-special,fgp, fpp, fgs, gm1, mohajer-zuo-paa, ire, ct}; we recall their main construction in section \ref{sec:colemanoort}. The classification of such examples in case of cyclic covering of $\PP$ is due to Moonen (\cite{moonen-special}). The complete list of all currently known examples of special families of Galois coverings is provided in \cite{fgp, fpp}. Our second result provides new evidence for the Coleman-Oort conjecture.


    

%

\begin{mainteoB}
For $g\geq 8$ there are no totally geodesic (in particular, no special) families of cyclic coverings with no repeating factors in the monodromy. 
\end{mainteoB}    
(See Theorem \ref{classificazione}). More precisely, we prove that the only positive dimensional totally geodesic families of cyclic coverings with no repeating factors in the monodromy are among the known ones (namely, among the ones appearing in \cite{fgp, fpp}).  In other words, we provide a classification of all totally geodesic families of cyclic coverings with no repeating factors in the monodromy for all $g'\geq 0$. 
Theorem B follows combining Theorem A with the fact that the smallest totally geodesic subvariety containing the moduli image of a family is given by (the image in $\A_g$) of the orbit of $\Mon^{0,ad}_{\R}$.
We point out that, along with the proof, we also show that families with no repeating factors in the monodromy are special if and only if they are totally geodesic (see Proposition \ref{necessità}). Moreover, we observe that, even though assuming no repeating factors appears as a strong condition on the monodromy, we have:

\begin{prop*}
All known examples of special families of cyclic coverings (namely, the ones appearing in \cite{fgp, fpp}) have no repeating factors in the monodromy.
\end{prop*}
(See Theorem \ref{no repeating moonen}). 
\end{say}

\begin{say}
    Simple examples of families with no repeating factors in the monodromy are provided by families of double and triple coverings. Let $\mathcal{B}_{g,h}$ be the image in $\M_g$ of the family of all double covers $f: C\ra C'$ with $g(C)=g$ and $g(C')=h$. Likewise, let $\mathcal{T}_{g,h}(r,m)$ be the image in $\mg$ of the family of all Galois  triple coverings $f:C\ra C'$ with $g(C)=g,\ g(C')=h $, $r=2m+n$ ramification points and fixed monodromy datum $
  \Theta=(x_1, \ldots, x_{2m}, y_1, \ldots, y_n)$, where $x_i+x_{i+1}=0$ mod $3$ and $y_i+y_j\neq 0$ mod $3$ for all $i$ and $j$. Our third result is: 
  \begin{mainteoC}The double covering locus $\mathcal{B}_{g,h}$ is not totally geodesic unless $$(g,h)= (1,0), (2,0), (2,1), (3,1).$$
  The triple covering locus $\mathcal{T}_{g,h}(r,m)$ is not totally geodesic unless 
  $$(g,h,r,m)=(2,0,4,2), (3, 0, 5, 1), (4,0,6,0), (3,1, 2, 1), (4,1,3,0).$$
  \end{mainteoC}
 (See Corollaries \ref{double}, \ref{teo trig non ram}, \ref{triple-with-ram}). The case of the bielliptic locus $\mathcal{B}_{g,1}$ was already considered in \cite{fp}, where, with a completely different technique, the authors show that it is not totally geodesic for $g\geq 4$ (while it is for $g=3$). Our Theorem $C$ extends this result to all double and triple covering loci.

 The paper is organized as follows: section \ref{sec:preliminaries} recalls some basic facts on families of coverings, section \ref{sec:admissibleandprym} discusses a result on Prym varieties of certain admissible coverings, section \ref{sec:maxmonodromy} proves Theorem A, while Theorem B is shown in section \ref{sec:colemanoort}. Finally, in section \ref{sec:doubleandtriple} we study families of double and triple coverings. \\
\end{say}

{\bfseries \noindent{Acknowledgements.}} We thank Matteo Costantini, Juan Carlos Naranjo, and Andr\'es Rojas for useful comments.

\section{Preliminaries on families of $G$-curves}\label{sec:preliminaries}
Let $G$ be a finite group and let $C, \widetilde{C}$ be smooth projective complex curves equipped with a $G$-action.  
We say that $C$ and $\widetilde{C}$ have the same \emph{topological type} (or are \emph{topologically equivalent}) if there exists an automorphism $\eta \in \operatorname{Aut}(G)$ together with an orientation–preserving homeomorphism
\[
f: C \to \widetilde{C}
\]
such that
\[
f(g\cdot x) = \eta(g)\cdot f(x), \qquad g \in G, \; x\in C.
\]
If, in addition, the map $f$ is biholomorphic, then $C$ and $\widetilde{C}$ are called \emph{$G$-isomorphic}.  

We denote by $\mathfrak{T}_{g',r}(G)$ the collection of topological types of $G$-actions for which the quotient $C':=C/G$ has genus $g'$ and the covering $C \to C'$ has exactly $r$ branch points.  
Given a fixed topological type $\tau \in \mathfrak{T}_{g',r}(G)$, there exists an algebraic family of curves with a $G$-action (see \cite{ganzdiez, GT2}),
\[
\pi : \mathcal{C}_\tau \longrightarrow \mathsf{B}_\tau ,
\]
with the following properties:
\begin{enumerate}
  \item every fiber in the family carries a $G$-action topologically equivalent to $\tau$;
  \item any curve endowed with a $G$-action of type $\tau$ is $G$-isomorphic to some fiber, and only finitely many fibers represent the same isomorphism class.
\end{enumerate}
In the following, we denote by  $\mathsf{M}_\tau$ the image of the family in  $\mathcal{M}_g$ and by $\Z_{\tau}$ the closure of its image in $\A_g$ under the Torelli map.

It is a classical fact that topological types of $G$-actions on curves can be described
combinatorially in terms of monodromy and branching data. More concretely, fix a compact
Riemann surface $C'$ of genus $g' \geq 0$ and an ordered $r$-tuple of distinct points
$t=(t_1,\dots,t_r)$ on $C'$. Let $U_t := C' \setminus \{t_1,\dots,t_r\}$ and choose a base point
$t_0 \in U_t$. Then the fundamental group $\pi_1(U_t,t_0)$ admits a canonical presentation
\[
\pi_1(U_t,t_0)\;\cong\; 
\Gamma_{g',r} \;=\;
\big\langle \alpha_1,\beta_1,\dots,\alpha_{g'},\beta_{g'},\gamma_1,\dots,\gamma_r \;\big|\;
\prod_{i=1}^r \gamma_i \cdot \prod_{j=1}^{g'} [\alpha_j,\beta_j] = 1 \big\rangle,
\]
where $\alpha_j,\beta_j$ are simple loops generating a symplectic basis of $H_1(C',\mathbb{Z})$,
and $\gamma_i$ is a small loop around the puncture $t_i$, defined with respect to a chosen
system of arcs $\tilde{\gamma}_i$ from $t_0$ to $t_i$.

Now suppose $f:C \to C'$ is a $G$-Galois covering branched over $t$. Writing
$V = f^{-1}(U_t)$, the restriction $f|_V : V \to U_t$ is an unramified Galois covering
with deck transformation group $G$. Hence there is a surjective homomorphism
\[
\pi_1(U_t,t_0) \twoheadrightarrow G,
\]
well-defined up to inner automorphisms of $G$. Equivalently, we obtain an epimorphism
\[
\theta : \Gamma_{g',r} \twoheadrightarrow G.
\]
If $m_i$ denotes the order of the local monodromy around $t_i$, we set
$m = (m_1,\dots,m_r)$, and we recall the following notion.

\begin{defin}
A \emph{datum} is a triple $(m,G,\theta)$, where
\begin{itemize}
\item $m = (m_1,\dots,m_r)$ is an $r$-tuple of integers $m_i \geq 2$;
\item $G$ is a finite group;
\item $\theta : \Gamma_{g',r} \twoheadrightarrow G$ is an epimorphism such that $\theta_i:=\theta(\gamma_i)$ has order $m_i$ for each $i$.
\end{itemize}
\end{defin}
The topological type of a $G$-action on a curve is thus encoded by such a datum:
the group $G$, the local monodromy orders $m_i$, and the epimorphism $\theta$.
As a consequence, the classification of topological types reduces to the combinatorics of generating systems of $G$ subject to these relations.

Finally, let us recall that for any fiber $C_b,\  b\in \B_\tau$ of the family, the group action induces an action of $G$ on holomorphic one-forms:
\[\rho :G\ra \Gl(H^0(C_b,K_{C_b})), \;  \rho(g)(\omega)=g.\omega:=(g^{-1})_*(\omega).\]
Notice that the equivalence class of $\rho$ does not depend on the point $b\in B_\tau$. The Chevalley-Weil formula computes the dimension $m_\chi$ of the $\chi$-eigenspace for the action of $G$ on $H^{0}(C, K_C)$, where $\chi$ is a $\C$-irreducible representation of $G$. In particular, for a family of cyclic coverings of order $d$, with monodromy $\Theta=(\theta_1, \ldots, \theta_r)$ and genus of the base curve $g'$, the formula reads as follows (\cite{cw}): 
\begin{equation}\label{Chevalley}
    m_{\chi_i} = (g' - 1) +\sum_{i=1}^r\frac{[i\theta_i]_d}{d}+\varepsilon,
\end{equation}
where $\chi_i$ is the representation associated with the eigenvalue $\xi^i$, $\xi$ is a primitive $d$-th root of the unity, $[a]_d$ denotes the unique representative of $a\in \mathbb{Z}/d\mathbb{Z}$ in $\{0, ...,d-1\}$, and $\varepsilon$ is $1$ when $\chi$ is the trivial representation and $0$ otherwise.

\section{An admissible covering and its Prym variety}\label{sec:admissibleandprym}
In this section, we consider admissible \(G\)-coverings of connected stable curves of arithmetic genus \(g'\) and we investigate their associated Prym varieties. We start by recalling the definition. 
\begin{defin}
    Let $\mathcal{X} \ra S$ be a family of connected nodal curves of arithmetic genus g and let $G$ be a finite group. A $G$-covering $f: \mathcal{Z}\ra \mathcal{X}$ is called a family of admissible $G$-coverings if it satisfies
    \begin{itemize}
        \item[1)]  the composition $\mathcal{Z}\ra \mathcal{X} $ is a family of nodal curves;
\item[2)] every node of a fiber of $\mathcal{Z}\ra S$ maps to a node of the corresponding fiber of
$\mathcal{X} \ra S$;
\item[3)] $\mathcal{Z}\ra \mathcal{X}$ is a principal $G$-bundle away from the nodes,
\item[4)] if the node $z$ lies over the node $x$ of the corresponding fibre of $\mathcal{X} \ra S$, and $z_1$ and $z_2$ are local coordinates of the two branches near $z$, then the generator $h$ of the stabilizer $Stab_G(z)$ acts as
\begin{equation}\label{gluing}
    h.(z_1, z_2)=(\zeta z_1, \zeta^{-1} z_2),
\end{equation}
where $d=\operatorname{ord}(h)$ and $\zeta $ is a primitive $d$-th root of the unity.
   \end{itemize}
\end{defin}
When $S = \operatorname{Spec}(\mathbb{C})$, we say that $f: Z\ra X$ is an admissible $G$-cover. In this case, let $\tilde{f}: \tilde Z\ra \tilde{X}$ be the lifting of $f$ on the normalisations. Then for each node of $Z$, condition $(4)$ is a compatibility condition for the local monodromies of $\tilde{f}$ at the two points in $\tilde Z$ lying above the node in $Z$.

Now, let \(G\) be a cyclic group of order \(d\) with generator \(\sigma\). In the following, let
\[
f_D : \widetilde{D} \longrightarrow D,
\]
be the admissible $G$-cover satisfying the following:
\begin{itemize}
    \item the curves $\widetilde{D}$ and $D$ are connected stable curves of genus $g$ and $g'$, respectively;
    \item  the curve \(\widetilde{D}\) has exactly one node, and this node is fixed by the action of \(G\);
    \item the normalization of \(f_D\) is the \(G\)-cover
\[
f_D^{\nu} : \widetilde{D}^{\nu} \longrightarrow D^{\nu},
\]
with monodromy data $\Theta^{\nu} = (1, -1, \theta_1, \dots, \theta_r).$
\end{itemize}
More concretely, let \(p_1, p_2 \in D^{\nu}\) be the two preimages of the node of \(D\).  These points are the branch points of the normalized cover \(f_D^{\nu}\), with local monodromies \(1\) and \(-1\), respectively.
The admissible cover \(f_D\) is obtained from its normalization \(f_D^{\nu}\) by gluing the points \(p_1\) and \(p_2\) in \(D^{\nu}\) to form the node of \(D\), and simultaneously gluing their corresponding preimages in $\tilde D\nu$. Note that we have $g(\tilde D^{\nu})=g-1$ and $g(D^{\nu})=g'-1$. An example of such an admissible cover is presented in Figure \ref{fig Admissible 1}. \\
\vspace{1cm}

\tikzset{every picture/.style={line width=0.75pt}}\label{fig Admissible 1} 
\begin{figure}[h]
\begin{tikzpicture}[x=0.75pt,y=0.75pt,yscale=-1,xscale=1]

\draw    (132.8,227.6) .. controls (221.8,269.6) and (77.8,276.6) .. (159.8,229.6) ;
\draw    (32.8,178.6) .. controls (49.11,181.95) and (51.62,177.73) .. (52.8,171.6) .. controls (53.98,165.47) and (46.93,158.62) .. (31.8,165.6) ;
\draw    (126.8,166.6) .. controls (165.8,174.6) and (153.8,134.6) .. (127.8,146.6) ;
\draw    (127.8,180.6) .. controls (166.8,188.6) and (154.8,117.6) .. (128.8,129.6) ;

\draw    (177.7,143.7) .. controls (138.61,136.18) and (151.09,176.03) .. (176.95,163.71) ;
\draw    (176.53,129.72) .. controls (137.44,122.19) and (150.3,193.04) .. (176.15,180.72) ;

\draw    (159.8,229.6) .. controls (216.8,215.6) and (250.8,225.6) .. (259.8,229.6) ;
\draw    (5.8,227.6) .. controls (63.8,217.6) and (123.8,223.6) .. (132.8,227.6) ;
\draw    (75.8,162.81) .. controls (96.83,165.77) and (90.36,150.96) .. (76.34,155.4) ;
\draw    (76.34,168) .. controls (97.37,170.96) and (90.9,144.66) .. (76.88,149.11) ;

\draw    (30.8,145.6) .. controls (47.11,148.95) and (49.62,144.73) .. (50.8,138.6) .. controls (51.98,132.47) and (44.93,125.62) .. (29.8,132.6) ;
\draw    (238.8,179.6) -- (239.93,210) ;
\draw [shift={(240,212)}, rotate = 267.88] [color={rgb, 255:red, 0; green, 0; blue, 0 }  ][line width=0.75]    (10.93,-3.29) .. controls (6.95,-1.4) and (3.31,-0.3) .. (0,0) .. controls (3.31,0.3) and (6.95,1.4) .. (10.93,3.29)   ;
\draw    (421.8,230.6) -- (609.8,231.6) ;
\draw    (366.8,157.6) -- (320.8,158.56) ;
\draw [shift={(318.8,158.6)}, rotate = 358.81] [color={rgb, 255:red, 0; green, 0; blue, 0 }  ][line width=0.75]    (10.93,-3.29) .. controls (6.95,-1.4) and (3.31,-0.3) .. (0,0) .. controls (3.31,0.3) and (6.95,1.4) .. (10.93,3.29)   ;
\draw    (435.8,192.6) .. controls (452.11,195.95) and (454.62,191.73) .. (455.8,185.6) .. controls (456.98,179.47) and (449.93,172.62) .. (434.8,179.6) ;
\draw    (432.8,163.6) .. controls (449.11,166.95) and (451.62,162.73) .. (452.8,156.6) .. controls (453.98,150.47) and (446.93,143.62) .. (431.8,150.6) ;
\draw    (480.8,172.81) .. controls (501.83,175.77) and (495.36,160.96) .. (481.34,165.4) ;
\draw    (481.34,178) .. controls (502.37,180.96) and (495.9,154.66) .. (481.88,159.11) ;

\draw    (520.8,173.98) .. controls (544.83,177.4) and (537.43,160.32) .. (521.42,165.44) ;
\draw    (521.42,179.96) .. controls (545.44,183.38) and (538.05,153.06) .. (522.03,158.18) ;

\draw    (578.77,164.73) .. controls (554.7,161.66) and (562.34,178.63) .. (578.28,173.28) ;
\draw    (578.07,158.76) .. controls (554,155.7) and (561.83,185.9) .. (577.77,180.55) ;

\draw  [dash pattern={on 0.84pt off 2.51pt}]  (536.8,166.6) -- (536.8,227.6) ;
\draw [shift={(536.8,229.6)}, rotate = 270] [color={rgb, 255:red, 0; green, 0; blue, 0 }  ][line width=0.75]    (10.93,-3.29) .. controls (6.95,-1.4) and (3.31,-0.3) .. (0,0) .. controls (3.31,0.3) and (6.95,1.4) .. (10.93,3.29)   ;
\draw  [dash pattern={on 0.84pt off 2.51pt}]  (561.8,167.6) -- (561.8,228.6) ;
\draw [shift={(561.8,230.6)}, rotate = 270] [color={rgb, 255:red, 0; green, 0; blue, 0 }  ][line width=0.75]    (10.93,-3.29) .. controls (6.95,-1.4) and (3.31,-0.3) .. (0,0) .. controls (3.31,0.3) and (6.95,1.4) .. (10.93,3.29)   ;

\draw (251,183) node [anchor=north west][inner sep=0.75pt]   [align=left] {d:1};
\draw (527,236) node [anchor=north west][inner sep=0.75pt]   [align=left] {p1};
\draw (554,236) node [anchor=north west][inner sep=0.75pt]   [align=left] {p2};

\end{tikzpicture}
\caption{}{Example of a $4:1$ admissible cover (left), where the sheets are glued according to \eqref{gluing}, and its normalization (right).    }
	\label{fig Admissible 1}
\end{figure}
\vspace{1cm}

\noindent{We now study the Prym variety associated with the covering $f_D$, inspired by \cite{beauville, faber, langeortega}.}
At the node $s$ of $\tilde D$, we make the usual identification:
\[\mathcal{K}^*_s/\OO_s^*\cong \C^*\times \Zeta\times \Zeta,\]
and the induced action by $\sigma$ works as:
\[\sigma^*((z, m,n)_s)=(\rho^{(m+n)}z, m,n)_s,\]
where $\rho$ is a primitive $d$-th root of the unity. We start with the following:
\begin{lemma}\label{line bundles in P}
    Let $L\in \Pic(\tilde D)$ such that $\Nm(L)\simeq \OO_{D} $. Then $$L\simeq M\otimes \sigma^*M^{-1},$$
    for some $M\in \Pic(\tilde D)$.
    \begin{proof}
         We need to describe divisors $H$ such that $L\simeq \OO_{\tilde D}(H)$ and $(f_D)_*(H)=0$. Writing $H= \sum_{x\in \tilde D_{reg}}x+(z, m,n)_s$, we get that $H$ must be a linear combination of two kinds of contributions:  
from smooth points $x\in \tilde D_{\mathrm{reg}}$, divisors of the form $x-\sigma^j x$ ($j\in [0,\ldots, d-1]$), and at the singular point $s\in \tilde D$, divisors of the form $(\rho^k,0,0)_s$. Indeed, using that
        \begin{gather*}
            (f_D)_*((z,m,n)_s)=(\Pi_{i=0}^{d-1}(\sigma^i)^*z, m, n)_s=\begin{cases}
                (z^d, m,n)_{f(s)}\quad\quad\quad \text{if $d$ odd}\\
                ((-1)^{(m+n)}z^d, m,n)_{f(s)}\quad \text{if $d$ even}
            \end{cases} 
        \end{gather*}
        we have that  
        \[(f_D)_*((\rho^k, 0, 0)_s)=(1,0,0)_{f(s)}.\]
        Since $(\rho^k, 0, 0)_s=\sum (\rho, 0, 0)_s$, it suffices to show that $(\rho, 0, 0)_s$ is in the image of $1-\sigma^*.$ This follows from \[(1,-1,0)-\sigma^*(1,-1,0)=(\rho, 0, 0).\]
        Hence $H= E-\sigma^*E, $ for some divisor $E\in \tilde D$. 

        
    \end{proof}
\end{lemma}
Let $P$ denote the kernel of the norm map $\Nm: J\tilde D\ra JD$. Lemma \ref{line bundles in P} implies that $P$ is the variety of line bundles in $\ker\Nm$ of the form $M\otimes\sigma^*M\meno$. The Prym variety $P(\tilde D, D)$ of $f_D$ is, by definition, the connected component of the identity of $P$. We have the following:
\begin{teo}\label{iso prym}
    The Prym variety $P(\tilde D, D )$ is isogenous to the Prym variety $P(\tilde D^{\nu}, D^{\nu} )$ associated with the normalized covering $\tilde D^{\nu}\ra D^{\nu}$. 
    \begin{proof}
        We consider the following diagram of commutative algebraic groups:
\begin{equation}\label{diagramma prym}
            \xymatrix{ 0\ar[r] &\tilde T\ar[d]_{\Nm}\ar[r] & J\tilde D\ar[d]_{\Nm}\ar[r]& J\tilde D^{\nu}\ar[d]_{\Nm}\ar[r] &0\\
        0\ar[r] & T\ar[r] & JD\ar[r]& JD^{\nu}\ar[r] &0
}
\end{equation}
where $\tilde T, T$ are the groups of classes of divisors of degree 0 and singular support. Since $(f_D)^*$ is injective on $T$ and $\Nm\circ (f_D)^*=d$, the norm on $\tilde T$ is surjective and $$\ker \Nm\restr{\tilde T}\cong \tilde T_d:=\{\text{points in $\tilde T$ whose order divides $d$}\}.$$ Hence, taking the kernels in diagram \eqref{diagramma prym}, we obtain the exact sequence: 
\begin{equation}\label{successione esatta Prym}
    0\ra \tilde T_d\ra P\ra R\ra 0,
\end{equation}
where $R:=\ker \Nm J\tilde D^{\nu}\ra JD^{\nu}$. Since $\tilde D^{\nu}\ra D^{\nu}$ is ramified, $R$ is an abelian variety, and indeed, it is the Prym variety $P(\tilde D^{\nu}, D^{\nu} ).$ Moreover, since $\dim \tilde T=\dim T=1$, we have that $\tilde T_d\cong \Zeta/d$. Therefore, using \eqref{successione esatta Prym}, we get that $P$ and $R$ are isogenous.  
    \end{proof}
\end{teo}

\section{Maximal monodromy}\label{sec:maxmonodromy}
In this section we prove a criterion for the maximality of the algebraic monodromy group for certain families of cyclic coverings. We start by describing the natural $\mathbb{Q}$-variation of Hodge structures associated with a family of Galois coverings.

\begin{say}
Let $G$ be a finite cyclic group. As explained in Section~\ref{sec:preliminaries}, let us consider a topological type of  $G$-action and let $f:\mathcal{C}\rightarrow B$ be the corresponding algebraic family of curves parametrizing all Galois covers carrying a $G$-action topologically equivalent to the one fixed. Thus every cover
$C_b \to C'_b = C_b/G$ has the prescribed invariants: the genera $g=g(C_b)$ and
$g'=g(C'_b)$, together with the chosen ramification and monodromy data.
Let
\[
\V = R^1f_*\QQ
\]
be the local system underlying the natural $\QQ$-variation of Hodge structure (VHS)
associated with this family. Fix a Hodge-generic point $b\in B$ with respect to the variation $\V$.
The $G$-action on the family induces a compatible action on the VHS.
Indeed, denote by $X(G)$ the set of $\C$-irreducible characters of $G$. Then the complexified
local system decomposes as
\[
\V_{\C} = \bigoplus_{\chi\in X(G)} \V_{\chi},
\]
where $\V_{\chi}$ is the $\chi$-isotypical summand. At the fiber $b$, we have
$\V_b = H^1(C_b,\QQ)$, and after extension of scalars to $\C$ we obtain
\[
H^1(C_b,\C) = H^{1,0}(C_b) \oplus H^{0,1}(C_b).
\]
Moreover, the space of holomorphic differentials admits a decomposition
\begin{equation}\label{H0kC}
    H^{1,0}(C_b) = \bigoplus_{\chi\in X(G)} m_{\chi} V_{\chi},
\end{equation}
where $V_{\chi}$ denotes the irreducible $\C$-representation corresponding to $\chi$
and $m_{\chi}$ its multiplicity. In particular,
\[
\V_{\chi,b} = H^1(C_b,\C)_{\chi}
= m_{\chi}V_{\chi}\;\oplus\;\overline{m_{\chi}V_{\chi}}
= m_{\chi}V_{\chi} \oplus m_{\bar\chi}V_{\bar\chi}.
\]

Let $X(G,\R)$ be the set of $\R$-irreducible representations of $G$. For a character
$\chi\in X(G)$, we denote by $(\chi,\overline{\chi})$ the corresponding real
representation. Then one has a decomposition over $\R$ of the form
\[
\V_{\R} = \bigoplus_{(\chi,\overline{\chi}) \in X(G,\R)} \V_{(\chi,\overline{\chi})},
\]
and if $\chi\neq\overline{\chi}$, we have
\[
\V_{(\chi,\overline{\chi})}\otimes_{\R}\C = \V_{\chi} \oplus \V_{\overline{\chi}},
\]
whereas if $\chi=\overline{\chi}$ one obtains
\[
\V_{(\chi,\chi)}\otimes_{\R}\C = \V_{\chi}.
\]

The polarization on the VHS restricts as follows: when $\chi\neq \overline{\chi}$,
it induces a Hermitian form of signature $(m_\chi, m_{\overline{\chi}})$ on
$\V_{(\chi,\overline{\chi})}$; in the case $\chi=\overline{\chi}$, it yields a
symplectic form on $\V_{(\chi,\chi)}$
(see \cite[Corollary~2.21--2.23]{dm}).

Let $\Mon^0 \subset \operatorname{GL}(H^1(C_b,\QQ))$ denote the connected monodromy group, and
$\Hdg \subset \operatorname{GL}(H^1(C_b,\QQ))$ the generic Hodge group of the VHS.
By a theorem of Andr\'e \cite[Theorem~1]{andre}, $\Mon^0$ is a semisimple normal subgroup
of $\Hdg^{der}$. More precisely, we have
\begin{equation}\label{inclusioni mon hodge}
    \Mon^0 \;\subseteq\; \Hdg^{der} \;\subseteq\; \ (\operatorname{Sp}(H^1(C_b,\QQ))^G)^{der},
\end{equation}
where $\Sp(H^1(C_b,\QQ))^G$ is the centralizer of $G$ inside $\Sp(H^1(C_b,\QQ))$.
Upon extension of scalars to $\R$, the latter admits the following explicit
decompositions:
\begin{align}\label{decomposizione SG}
		\Sp(H^1(C_b, \R))^G &\simeq \prod_{(\chi, \bar\chi )\in X(G, \R)} \Sp(H^1(C_b, \R)_\chi)^G \\ &\simeq  \prod_{(\chi, \bar\chi ), \chi\neq\bar\chi} \U(m_\chi, m_{\bar \chi})\times\prod_{(\chi, \bar\chi ), \chi=\bar\chi}\Sp(2m_\chi, \R).\notag
	\end{align}
 Thus, we have 
	\begin{gather}\label{prodmon}
		\Mon^{0}_\R\hookrightarrow (\operatorname{Sp}^G)^{der}=\prod_{\chi\neq \bar{\chi}} \SU(m_\chi, m_{\bar\chi}) \times \prod_{\chi=\bar\chi} \Sp(2m_\chi, \R).
	\end{gather}
    We denote by $\Mon^0(\chi)$ the projection of $\Mon^0_\R$ onto the $\chi$-factor of the decomposition.  If $\chi\neq \bar \chi$, we have $\Mon^0(\chi)\subseteq \SU(m_\chi, m_{\bar \chi})$, whether for $\chi=\bar \chi$, then $\Mon^0(\chi)\subseteq \Sp_{2m_\chi}.$
\end{say}   

\begin{say}\label{degenerazione}
For the remaining part of this section, we focus on the inclusion \eqref{prodmon} and we prove a criterion for the maximality of the monodromy. We will work with families satisfying the following condition:
\begin{defin}
    The family has \textbf{no repeating} factors in the monodromy if for all $\chi_i, \chi_j\in X(G),  i\neq j$, different from the trivial representation and with  $m_{\chi_i}m_{\bar\chi_i}m_{\chi_j}m_{\bar\chi_j} \neq 0$, we have that: \begin{itemize}
     \item 
		$\SU(m_{\chi_i}, m_{\bar\chi_i})\not\cong \SU(m_{\chi_j}, m_{\bar\chi_j}) $ 
    if $\chi_i\neq \bar\chi_i$ and $\chi_j\not\cong\bar\chi_j$
    \item $\Sp(2m_{\chi_i}, \R)\neq \Sp(2m_{\chi_j}, \R)$ if $\chi_i=\bar\chi_i$ and $\chi_j=\bar\chi_j$.
 \end{itemize}
 \end{defin}
 From now on we set \begin{equation}\label{sp pos}
     	(\Sp(H^1(C_b, \R))^G)^{pos}:=	\prod_{\chi\neq \bar{\chi},\, m_\chi m_{\bar\chi}\neq 0} \SU(m_\chi, m_{\bar\chi}) \times \prod_{\chi=\bar\chi, \,m_\chi\neq 0} \Sp(2m_\chi, \R).
 \end{equation}
 Hence, a family has no repeating factors in the monodromy if in $(\Sp(H^1(C_b, \QQ))^G)^{pos}$ all the factors are mutually distinct.
Our main theorem uses an induction argument on the genus $g'$ of the base curve $C'_b$. The base of the induction, i.e. the case with $g'=0$, is the following result:

\begin{teo}\label{tutti diversi}
   	Let $\mathcal{C}\rightarrow \B$ be a family of cyclic $G$-coverings of the line with no repeating factors in the monodromy. Then the monodromy satisfies
    \begin{gather*}
         (\Sp(H^1(C_b, \R))^G)^{pos} \subseteq \Mon^{0}_{\R} .
    \end{gather*}   
\end{teo}
\begin{proof}
    This result is proved in \cite[Theorem 1.7]{ST}, for the reader convenience we sketch the proof. The first ingredient, due to Rohde \cite{rohde}, is that for cyclic coverings of the line, whenever the multiplicities $m_\chi$ and $m_{\bar\chi}$ are nonzero, one has $\Mon^0(\chi) = \SU(m_\chi, m_{\bar \chi})$ if $\chi \neq \bar \chi$, and $\Mon^0(\chi) = \Sp_{2m_\chi}$ if $\chi = \bar \chi$. 
    
    By definition $\Mon^0_{\R}$ is a semisimple group contained in $(\Sp(H^1(C_b, \R))^G)^{der}$, and by what recalled above, it projects surjectively onto each factor of $(\Sp(H^1(C_b, \R))^G)^{pos}$. Thus, since we are assuming that no repeating factor in $(\Sp(H^1(C_b, \R))^G)^{pos}$ appears, no diagonal embedding is allowed, and thus we get the claimed inclusion. 
    
\end{proof}

The induction step in our proof is based on a degeneration argument. More precisely, we consider admissible $G$-coverings that sit in the boundary of the family $f: \mathcal{C}\rightarrow \B$. The general element in the family is a covering $C \ra C'$ with $g(C) = g$, $g(C') = g'$ and monodromy $\Theta = (\theta_1, \ldots , \theta_s).$ We degenerate this covering to an admissible $G$-cover $f_D: \tilde D \ra D$ of a connected, stable, 1-nodal curve $D$ of arithmetic genus $g'$, whose covering $\tilde D$ has exactly one node on which is fixed by $G$ (and mapped to the node of $D$). The normalization of $f_D$ is the covering  $\tilde D^{\nu} \ra D^{\nu}$ with monodromy $\Theta^{\nu}=(1, -1, \theta_1, \dots, \theta_s)$. This is the setup of section \ref{sec:admissibleandprym}.
\begin{remark}\label{degeneration}
Note that the covering $\tilde D \ra D$ is the generic element of an irreducible component of the intersection $\overline{\mathsf{M}}\cap \Delta_0, $ where $\mathsf{M}$ is the moduli image of the family $f$ in $\mg$, $\overline{\mathsf{M}}$ its closure in $\overline{\M}_g$, and $\Delta_0$ the boundary divisor of irreducible nodal curves. We will refer to the family of the coverings $\tilde D \ra D$ as the $\Delta_0$-\textit{degeneration} of $f$. 
\end{remark}
As in \eqref{H0kC}, let $m_\chi$ be the multiplicity of the $\C$-irreducible representation corresponding to the character $\chi$ of the group $G$ acting on $H^{1,0}(C)$. Similarly, let $m_{\chi}^{\nu}$ be the dimension of the $\chi$-eigenspace for the action of $G$ on $H^{1,0}(\tilde D^{\nu})$. The Lemma below compares the dimensions of the eigenspaces. 
\end{say}

\begin{lemma}\label{lemma-m}
     The following equality holds \begin{gather*}
        m_{\chi_i}=m_{\chi_i}^{\nu}
    \end{gather*}
    for all $\chi\in X(G)$ different from the trivial representation.
    \end{lemma}
\begin{proof}
    The result follows directly from the Chevalley-Weil formula \eqref{Chevalley}. Indeed, using that $\Theta^\nu$ is defined with the same $s$-tuple as $\Theta$ plus the pair $(-1,1)$, we have
    \[m^{\nu}_{\chi_i}=-1+m_{\chi_i}+1.\]
\end{proof}
We are now ready to state our main theorem.
\begin{teo}\label{main}
 Let $\mathcal{C}\rightarrow \B$ be a family of cyclic $G$-coverings with no repeating factors in the monodromy. Then the monodromy satisfies 
    \begin{gather}\label{formula1}
         (\Sp(H^1(C_b, \R))^G)^{pos} \subseteq \Mon^{0}_{\R}.
    \end{gather}      
    Moreover, if $g'\geq 1$, the monodromy is maximal, namely
    \begin{gather}\label{formula2}
        \Mon^{0}_{\R} = (\Sp(H^1(C_b, \R))^G)^{der}.
    \end{gather}  
\end{teo}
\begin{proof}
    Let $C_b\rightarrow C'_b=C_b/G$ be the generic element in the family. We use induction on the genus $g'$ of the base curve. For $g'=0$, the statement follows by Theorem \ref{tutti diversi}. Let us show that the statement holds for $g'>0$, assuming that it is true for $g'-1$. Since $JC_b\sim JC_b'\times \operatorname{Prym}(C_b, C_b')$, we have 
    \begin{gather*}
       \Mon^0(JC_b)\subseteq \Mon^0(JC'_b)\times \Mon^0(\operatorname{Prym}(C_b, C_b'))\subseteq\operatorname{Sp}_{2g'}\times \prod_{\chi\neq \bar{\chi}} \SU(m_\chi, m_{\bar\chi}) \times \prod_{\chi=\bar\chi} \Sp(2m_\chi, \R), 
    \end{gather*}
    where the first term in the product comes from the trivial representation. As $C_b'$ moves in the whole moduli space $\M_{g'}$, we have $\Mon^0(JC'_b)=\operatorname{Sp}_{2g'}$. Hence, it is enough to focus on $\Mon^0(\operatorname{Prym}(C_b, C_b'))$. Using a degeneration, we can show that \begin{gather}\label{upperbound}\prod_{\chi\neq \bar{\chi},  m_{\chi} m_{\bar {\chi}}\neq 0} \SU(m_\chi, m_{\bar\chi}) \times \prod_{\chi=\bar\chi,  m_\chi\neq 0} \Sp(2m_\chi, \R)\subseteq \Mon^{0}(\operatorname{Prym}(C_b, C_b')). \end{gather}
    Indeed, as in \ref{degenerazione}, we degenerate to the admissible cover $\tilde D \ra D$ whose normalization  $\tilde D^{\nu} \ra D^{\nu}$ is endowed with monodromy $\Theta^{\nu}=(1, -1, \theta_1, \dots, \theta_s)$. By Theorem \ref{iso prym}, we have $\operatorname{Prym}(\tilde D, D)\sim \operatorname{Prym}(\tilde D^{\nu},D^{\nu})$. By Lemma \ref{lemma-m}, the covering $\tilde D^{\nu}\ra D^{\nu}$ has \begin{gather*}
		\SU(m^\nu_{\chi_i}, m^\nu_{\bar\chi_i})\neq \SU(m^\nu_{\chi_j}, m^\nu_{\bar\chi_j})
	\end{gather*}
	for all $\chi_i, \chi_j\in X(G)$ different from the trivial representation, $i\neq j$ with $m^\nu_{\chi_i} m^\nu_{\bar {\chi}_i}\neq 0$. In  other words, $\tilde D^{\nu}\ra D^{\nu}$ also has no repeating factors in the monodromy. Since $g(D^\nu) = g' - 1$, the induction hypothesis implies that $$\prod_{\chi\neq \bar{\chi}, \\ \ m_{\chi} m_{\bar {\chi}}\neq 0} \SU(m_\chi, m_{\bar\chi}) \times \prod_{\chi=\bar\chi, \\ m_\chi\neq 0} \Sp(2m_\chi, \R)\subseteq \Mon^0(\operatorname{Prym}(\tilde D^{\nu}, D^{\nu})).$$  Moreover we have $\Mon^{0}(\operatorname{Prym}(\tilde D^{\nu},D^{\nu}))=\Mon^{0}(\operatorname{Prym}(\tilde D, D))\subseteq \Mon^{0}(\operatorname{Prym}(C_b, C_b'))$, which then implies \eqref{upperbound}, i.e.  \eqref{formula1} holds. 
    
    To complete the proof, we observe that, using Chevalley-Weil formula \ref{Chevalley}, no 0-dimensional eigenspace appears if $g'\geq 2$. Thus,  
    \begin{gather}\label{pos-der}
     (\Sp(H^1(C_b, \R))^G)^{der}= (\Sp(H^1(C_b, \R))^G)^{pos}\subseteq \Mon^0_{\R} \subseteq(\Sp(H^1(C_b, \R))^G)^{der},    
    \end{gather}
    yielding the equality \eqref{formula2} for all $g'\geq 2$. On the other hand, when $g'=1$, again Chevalley-Weil formula implies that, if there exists a representation $\chi_i$ with $m_{\chi_i}=0$, then $m_{{\bar\chi}_i}=0$ as well. This implies that the $\chi_i$-factor in $\operatorname{Sp}^G$ is trivial. Hence, the equality \eqref{pos-der} holds when $g'=1$ too, and thus we conclude.  
\end{proof}

\begin{remark}\label{g'=0}
It can be shown that the maximality of the monodromy implies the Putman-Wieland conjecture. This concerns families of Galois coverings $\mathcal{C}\ra B$ where $B\ra \mathcal{M}_{g',r}$ is dominant, and predicts that $\text{Jac}_{\mathcal{C/B}}$ does not have an isotrivial isogeny factor (i.e. no fixed part). As a consequence, we observe that the bound $g'\geq 1$ in Theorem \ref{main} for the equality in \eqref{formula2} is sharp. In particular, when $g'=0$ the inclusion \eqref{formula1} is the best we can hope for. This comes from the fact that, when $g'=0 $ there do exist examples of families with a fixed part in the Jacobians, hence, $\Mon^0_{\R}\neq (\Sp(H^1(C_b, \QQ))^G)^{der}$. For instance, the Jacobians of the family  \[C_b: y^4=x(x-1)(x-t), \; t\in \PP\smallsetminus\{0, 1, \infty\},\] 
 up to isogeny, have a constant elliptic factor. 

\end{remark}

\section{The Coleman-Oort conjecture}\label{sec:colemanoort}
In this section, we discuss some consequences of our main result in relation to the Coleman-Oort conjecture. 
\begin{say}

As recalled in the introduction, for low genus, there exist examples of positive dimensional special subvarieties $S\subset \mathcal{A}_g$ that lie generically in the Torelli locus. The main construction proceeds as follows. As in section \ref{sec:preliminaries}, let $G$ be a finite group, and consider the family $\mathcal{C} \to \B$ of all Galois covers with a fixed topological type, set  $\mathsf{M} \subset \mathcal{M}_g$ to be the image of the base $\B$ under the moduli map, and $\Z := \overline{j(\mathsf{M})} \subset \mathcal{A}_g$ as the closure of the image under the Torelli map.

Let $S(G) \subset \mathcal{A}_g$ denote the PEL-type special subvariety associated with the group $G$. This is the special subvariety whose uniformizing symmetric space corresponds to the group $(\Sp_{2g})^G$. By construction, we have the inclusion $\Z\subseteq S(G)$. It is shown in \cite{moonen-special} for the case $g' = 0$ and $G$ cyclic, and extended in \cite{fgp, fpp} to arbitrary $G$ and $g'$, that if the following numerical condition holds for a general $[C] \in \mathsf{M}$:
\begin{equation}
\label{star}
\tag{$\star$}
\dim \Z = \dim H^0(C, 2K_C)^G = \dim \left(S^2 H^0(C, K_C)\right)^G = \dim S(G),
\end{equation}
then $\Z$ is a special subvariety of $\mathcal{A}_g$. Since $\Z$ lies in the Torelli locus, this gives a counterexample to the Coleman-Oort conjecture whenever \eqref{star} is satisfied.
\end{say}

\begin{say}

It is important to note that the condition \eqref{star} is \emph{a priori} only sufficient: if $\dim \Z < \dim S(G)$, one cannot conclude that $\Z$ is not special. Indeed, if we denote by $S_f$ be the smallest special subvariety of $\mathcal{A}_g$ containing $\Z$, we have
\[
\Z \subseteq S_f \subseteq S(G),
\]
and, by definition, $\Z$ is special if and only if $\Z = S_f$. Proving the \emph{necessity} of condition \eqref{star} thus amount to showing that $\Z$ is special if and only if $\Z=S_f = S(G)$.  
\end{say}

\begin{say}
    To complete the picture, we recall that special subvarieties are subvarieties in $\A_g$, which are linear with respect to the locally symmetric metric induced from the Siegel space $\siegg$. More precisely, denote by $\pi: \sieg_g \ra A_g$ the natural projection map. An algebraic subvariety $Z\subset \A_g$ is totally geodesic if $Z=\pi(X)$, for a totally geodesic submanifold $X\subset \sieg_g$. Then,  special subvarieties are totally geodesic ({and not all totally geodesic subvariety of $\mathcal{A}_g$ are special}). Thus, if we denote by $X_f$ the smallest totally geodesic subvariety of $\A_g$ containing our family $\Z$, we have:
\[
\Z \subseteq X_f \subseteq S_f \subseteq S(G).
\]
We recall that $X_f$ is the totally geodesic subvariety
whose uniformizer is the symmetric space associated with the real group $\Mon^0_{\R}$, while $S_f$ is the special subvariety associated with the generic Hodge group $\Hdg_\R$ of the family (see \cite{moonen-linearity-1}). Hence, the inclusions above reflect the inclusions in \eqref{inclusioni mon hodge}.
\end{say}
In this setting, our Theorem \ref{main} implies: 
\begin{prop}\label{necessità}
   Let $\mathcal{C}\rightarrow \B$ be a family of cyclic $G$-coverings with no repeating factors in the monodromy. Then $X_f=S_f=S(G)$. In particular, $\Z$ is totally geodesic if and only if it is special and if and only if \eqref{star} holds.
\end{prop}
\begin{proof}
     The special subvariety $S(G)$ is the orbit of $\operatorname{Sp}^G$. This is the same as the orbit of $(\operatorname{Sp}^G)^{ad}$. Since $(\operatorname{Sp}^G)^{ad}=((\operatorname{Sp}^G)^{der})^{ad}$, from \eqref{prodmon}, we have $$(\operatorname{Sp}^G)^{ad}=\prod_{\chi\neq \bar{\chi}} \operatorname{PSU}(m_\chi, m_{\bar\chi}) \times \prod_{\chi=\bar\chi} \operatorname{PSp}(2m_\chi, \R). $$
Let $Q$ be one of the simple factors and denote by $\delta(Q)$
the dimension of the corresponding symmetric space. Note that when $Q$
is compact we  have $\delta(Q)=0$.
Therefore, \begin{align*}
    \dim S(G) &=\sum_\chi \delta( \operatorname{PSU}(m_\chi, m_{\bar\chi})) + \sum_\chi \delta(\operatorname{PSp}(2m_\chi, \R))\\
    &=\sum_{m_\chi m_{\bar\chi}\neq 0} \delta( \operatorname{PSU}(m_\chi, m_{\bar\chi}) + \sum_{m_\chi\neq 0} \delta(\operatorname{PSp}(2m_\chi, \R)).
\end{align*}  It follows that the dimension of the symmetric space obtained taking the orbit of $\operatorname{Sp}^G$ is the same as the one coming from $(\operatorname{Sp}^G)^{pos}$. Since $X_f$ is the totally geodesic subvariety associated with $\Mon^0_{\R}$, Theorem \ref{main} yields $X_f=S(G)$. Hence, the statement holds.
\end{proof}
This result allows us to go further in the classification of the counterexamples to the Coleman-Oort conjecture. Indeed in \cite{cgp}, employing a computational approach complemented by theoretical arguments, the authors proved that the positive-dimensional families of Galois covers satisfying
\eqref{star} with $2\leq g \leq 100$ are only those listed in \cite{fgp, fpp}. Our next result is an improvement of the classification in the case of families of cyclic coverings with no repeating factors in the monodromy. Indeed, we can prove the following:
\begin{teo}\label{classificazione}
    The only positive dimensional totally geodesic families of cyclic coverings with no repeating factors in the monodromy are the ones appearing in \cite{fgp, fpp}. In particular, for $g\geq 8$ there are no totally geodesic families of this kind. 
\end{teo}    
In order to prove this result, we first need a preliminary result concerning families of cyclic coverings of genus 1 curves. Our notation will be as in Remark \ref{degeneration}, namely $f: \mathcal{C}\rightarrow \B$ is a family of cyclic coverings with monodromy $\Theta=(\theta_1, \dots, \theta_s)$, $\tilde D \ra D$ is the family of $\Delta_0$-degeneration of $f$ and $\tilde D^{\nu} \ra D^{\nu}$ is the family of their normalisations. This is a family $f^{\nu}: \mathcal{D}^{\nu}\rightarrow \B^{\nu}$ with monodromy $\Theta^{\nu}=(1, -1, \theta_1, \dots, \theta_s)$.

\begin{prop}\label{genus1}
  Let $f: \mathcal{C}\rightarrow \B$ be a family of cyclic coverings of genus-$1$ curves with no repeating factors in the monodromy. Suppose that $f$ satisfies \eqref{star}, then  $f^{\nu}: \mathcal{D}^{\nu}\rightarrow \B^{\nu}$ satisfies it as well. 
\end{prop}
\begin{proof}
   First observe that $\dim \B^{\nu}=\dim \B-1$. We denote by $S^{\nu}(G)$ the PEL special subvariety associated with the family $f^{\nu}$. To conclude we need to show that $\dim S^{\nu}(G)=\dim S(G)-1$. We write $JC_b\sim JC_b'\times \Prym (C_b, C_b')$. Since \eqref{star} holds for the family $f$, we have
      \begin{align}\label{monod}
       \Mon^0(JC_b)&=\Mon^0(JC'_b)\times \Mon^0(\operatorname{Prym}(C_b, C_b'))\\ &=\operatorname{Sp}_{2}\times \prod_{\chi\neq \bar{\chi}} \SU(m_\chi, m_{\bar\chi}) \times \prod_{\chi=\bar\chi} \Sp(2m_\chi, \R). \notag
    \end{align}
   In particular, 
    \begin{gather*}
       \Mon^0(\operatorname{Prym}(C_b, C_b'))= \prod_{\chi\neq \bar{\chi}} \SU(m_\chi, m_{\bar\chi}) \times \prod_{\chi=\bar\chi} \Sp(2m_\chi, \R). 
    \end{gather*}
  By Lemma \ref{lemma-m}, the covering $D^{\nu}\ra D^{\nu}$ has still no repeating factions in the monodromy. Furthermore, by Theorem \ref{iso prym}, we have $\operatorname{Prym}(\tilde D, D)\sim  \operatorname{Prym}(\tilde D^{\nu},D^{\nu})$. Also, recall that $g(D^{\nu})=0$ and hence $J\tilde D^{\nu}\sim\operatorname{Prym}(\tilde D, D)$. Therefore, by Theorem \ref{tutti diversi}
   \begin{gather*}
       \Mon^0(J\tilde D^{\nu}) =\Mon^0(\operatorname{Prym}(C_b, C_b'))=\prod_{\chi\neq \bar{\chi}} \SU(m_\chi, m_{\bar\chi}) \times \prod_{\chi=\bar\chi} \Sp(2m_\chi, \R). 
   \end{gather*}
 Comparing with \eqref{monod} and taking the orbit spaces, we get that $\dim S^\nu(G)=\dim S(G)-1$.
\end{proof}
    \textit{Proof of Theorem \ref{classificazione}}
        Let $\mathcal{C}\rightarrow \B$ be a family of cyclic $G$-coverings with no repeating factors in the monodromy. By Proposition \ref{necessità}, the family is totally geodesic if and only if it is special and if and only if it satisfies condition \eqref{star}. In \cite{fgs}, it is shown that there are no special families satisfying \eqref{star} for $g'\geq 2$. On the other hand, the classification of such families with $g'=0$ follows from the main theorem in \cite{moonen-special}. All these examples are in genus $g\leq 7$. Hence, let us focus on the case $g'=1$. In \cite{fpp}, 4 examples are provided. Again, they all have $g\leq 7$. Therefore, we only have to prove that no other further examples exist. 
        Suppose that such a family $f$ exists.
        By Proposition \ref{genus1}, if follows that the normalization $f^{\nu}$ still satisfies \eqref{star} and hence it is one of the families listed in \cite{moonen-special}.
        Since $g(\tilde D^{\nu})=g(C)-1$, we obtain $g\leq 8$. Therefore, the statement follows from the main theorem in \cite{cgp}, where the authors prove that the only positive-dimensional families of Galois covers satisfying \eqref{star} with $2\leq g \leq 100$ are only those listed in \cite{fgp, fpp}.
        \qed\\

        \begin{remark}
            Notice that, thanks to Proposition \ref{necessità}, we may work with totally geodesic subvarieties rather than only with special ones. In other words, our results contribute to the stronger form of the Coleman–Oort conjecture, which predicts the non-existence of positive-dimensional totally geodesic subvarieties of $\ag$ that are generically contained in the Torelli locus.
        \end{remark}
We now show that, dealing with special families, our assumption on the monodromy is not too strong. Indeed, we have the following:
\begin{teo}\label{no repeating moonen}
    All known examples of special families of cyclic coverings have no repeating factors in the monodromy.
\end{teo}
\begin{proof}
    The examples of special families of cyclic coverings can be found in \cite{fgp, fpp}. The proof of the statement is based in the following considerations. First, if the family is one-dimensional, then the uniformizing symmetric space of $S(G)$ is $\sieg_1$, so this is automatic. Families with double and triple coverings always have no repeating factors in the monodromy; this is discussed in section \ref{sec:doubleandtriple}. Finally, for the remaining higher dimensional families we refer the reader to \cite[Theorem 1.2]{ct}, where it is shown that their uniformizing symmetric space is an irreducible symmetric space, and this implies that only one factor appears in $(\operatorname{Sp}^{G})^{pos}$ (hence, in particular, they have no repeating factors in the monodromy).
\end{proof}
\begin{say}
    It is interesting to compare our results with Theorem 1.3 in \cite{landesmaneco}. Indeed, in this paper, under some numerical assumptions, the authors show that the identity component of the Zariski closure of the image of the monodromy map associated with a family of $G$-covers is the is the commutator subgroup of $(\Sp_{2g})^G$. More precisely, in the case of a family of abelian covers $C\ra C'$ the result provides 
    \begin{equation*} \label{mon-der}
        \Mon^0_{\R}= (\Sp(H^1(C, \R))^G)^{der}
    \end{equation*}
    as soon as $g(C')\geq 4$. Hence, for $g'\geq 4$ our result provides only an alternative technique to get the maximality in our hypothesis. On the other hand, our result also covers the genera $1\leq g'\leq 4$ for cyclic coverings with no repeating factors in the monodromy. We also note that \cite[Theorem 1.3]{landesmaneco} says that, letting $d$ be the maximal dimension of an irreducible complex representation of any finite group $G$, then
    $$X_f=S_f=S(G)$$
    for every family of unramified $G$-covers with $g'\geq 2d+2$ or ramified $G$-covers with $g'>max\{2d + 1,d^2\}$.
    Thus, arguing as in the proof of Theorem \ref{classificazione} and, using that by \cite[Theorem 1.1]{fgs} if $g'\geq 2$ there are no families satisfying condition \eqref{star}, we get: 
\end{say}     

\begin{prop}
    Let $G$ be any finite group and let $d$ be the maximal dimension of an irreducible complex $G$-representation. Then, every family of unramified $G$-covers with $g'\geq 2d+2$ or ramified $G$-covers with $g'>max\{2d + 1,d^2\}$ is not totally geodesic. 
\end{prop}


\section{Application to double and triple coverings}\label{sec:doubleandtriple}
In this Section we apply Proposition \ref{necessità} and Theorem \ref{classificazione} to double and triple coverings of curves of a fixed genus. We start by fixing some notation. 
Let $\mathcal{B}_{g,h}$ be the image in $\M_g$ of the family of all double covers $f: C\ra C'$ with $g(C)=g$ and $g(C')=h$. Notice that, by Riemann-Hurwitz, the number of ramification points is $r=2g-4h+2$. It follows that $\operatorname{dim}(\mathcal{B}_{g,h})= 2g-h-1$. 

\begin{cor}\label{double}
The smallest totally geodesic subvariety  containing $\mathcal{B}_{g,h}$ is the PEL special variety $S(\mathcal{B}_{g,h})\cong\A_h\times \A_{g-h}.$ In particular, $\overline{j(\Bb_{g,h})}$ is not totally geodesic in $\A_g$ whenever $(g,h)\neq (1,0), (2,0), (2,1), (3,1)$. 
\end{cor}
\begin{proof}
    We start by observing that double coverings families obviously have no repeating factors in the monodromy, since the group $G=\mathbb{Z}/2\mathbb{Z}$ has only one non-trivial representation. Hence Proposition \ref{necessità} implies that the smallest special subvariety containg the moduli image of the family is the PEL $S(G)$. To compute it, we use the Chevalley-Weil formula and obtain that $$\Mon^0(\mathcal{B}_{g,h})=\Hdg_\R(\mathcal{B}_{g,h})=\operatorname{Hg}_\R(S(\mathcal{B}_{g,h})) = \operatorname{Sp}_{2h}\times \operatorname{Sp}_{2(g-h)}.$$ Thus,  $\Bb_{g,h}$ is totally geodesic if and only if $\dim\Bb_{g,h} =\dim \A_h\times \A_{g-h} $; namely if and only if 
    $(g,h)= (1,0), (2,0), (2,1), (3,1)$. 
\end{proof}

\vspace{0.5cm}
Now let $\mathcal{T}_{g,h}(r,m)$ be the image in $\mg$ of the family of all Galois  triple coverings $f:C\ra C'$ with $g(C)=g,\ g(C')=h $, $r=2m+n$ ramification points and fixed monodromy datum, which we write as: 
\begin{gather}\label{monodromia-canc}
  \Theta=(x_1, \ldots, x_{2m}, y_1, \ldots, y_n),  
\end{gather}
where $x_i+x_{i+1}=0$ mod $3$ and $y_i+y_j\neq 0$ mod $3$ for all $i$ and $j$. A couple $(x_i, x_{i+1})$ with $x_i+x_{i+1}=0$ mod $3$ is called a canceling pair. With the notation in \eqref{monodromia-canc}, $m$ is the number of the canceling pairs in $\Theta$, and $\sum_i y_i=0 \mod 3$.

We observe that (since it is a triple covering), the choice of $r$ and $m$ uniquely determines $\Theta$. In particular, we will assume that $(y_1, \ldots, y_n)= ([1]_3, \ldots, [1]_3)$, and thus $r=2m \mod 3$. Recall that, by Riemann-Hurwitz, we have $r=g-3h+2$.  We first consider the unramified case, i.e. we assume $r=0, h\geq 2$. We will denote by $\A(p, q)$ a totally geodesic subvariety of $\ag$ whose uniformizing symmetric is associated with the group $SU(p,q)$. We have the following: 
\begin{cor}\label{teo trig non ram}

The smallest totally geodesic subvariety containing $\mathcal{T}_{3h-2,h}$  is the PEL special variety $S(\mathcal{T}_{3h-2,h})\cong\A_h\times \A(h-1, h-1)$. In particular, $\overline{j(\mathcal{T}_{3h-2,h})}$ is not totally geodesic in $\A_g$.
\end{cor}
\begin{proof}
    As before, we observe that triple coverings families have no repeating factors in the monodromy,  as the group $G=\mathbb{Z}/3\mathbb{Z}$, has only one pair of non-trivial representations $(\chi_{[1]}, \chi_{[2]})$, which are conjugate to each other.  Thus, by Proposition \eqref{necessità}, the smallest special subvariety containing $\mathcal{T}_{3h-2,h}$ is $S(G)$. We use the Chevalley-Weil formula and obtain that $$\Mon^0(\mathcal{T}_{3h-2,h})=\Hdg_\R(\mathcal{T}_{3h-2,h})=\operatorname{Hg}_\R(S(\mathcal{T}_{3h-2,h})) = \operatorname{Sp}_{2h}\times \,SU(h-1,h-1) ,$$ which shows that $S(\mathcal{T}_{3h-2,h})$ is as claimed.
By dimension reasons, we get that $\mathcal{T}_{3h-2,h}\subsetneq S(\mathcal{T}_{3h-2,h}) $, and thus is not totally geodesic.
\end{proof}
Finally, we discuss triple Galois ramified coverings.  We have the following:
\begin{cor}\label{triple-with-ram}
The smallest totally geodesic subvariety containing $\mathcal{T}_{g,h}(r,m)$ is the PEL special variety $S(\mathcal{T}_{g,h}(r,m))\cong\A_h\times \A(d_1, d_2)$, where $d_1= h-1+ \frac{2}{3}r-\frac{1}{3}m$ and $d_2= h-1+ \frac{1}{3}r+\frac{1}{3}m$. 
Hence, $\overline{j(\mathcal{T}_{g,h}(r,m))}$ is totally geodesic in $\A_g$ if and only if $$(g,h,r,m)=(2,0,4,2), (3, 0, 5, 1), (4,0,6,0), (3,1, 2, 1), (4,1,3,0).$$
\end{cor}
\begin{proof}
  As before, we use Proposition \ref{necessità} and the Chevalley-Weil formula and obtain that $$\Mon^0(\mathcal{T}_{g,h}(r,m))=\Hdg_\R(\mathcal{T}_{g,h}(r,m)) = \operatorname{Sp}_{2h}\times SU(d_1,d_2),$$ with $d_1$ and $d_2$ as in the statement, which shows that $S(\mathcal{T}_{g,h})$ is as claimed. Since the smallest totally geodesic subvariety containing $\mathcal{T}_{g,h}(r,m)$ is the PEL, we have that the family is totally geodesic if and only if $\dim\mathcal{T}_{g,h}=\dim S(\mathcal{T}_{g,h}) $. When $h\geq 2$  this never happens by \cite{fgs}, while for $h=0,1$, this is satisfied if and only if it is one of the families in \cite{fgp,fp}, namely it is one of the family listed in the statement. 
\end{proof}

\end{document}